\documentclass[11pt,a4paper]{article}
\usepackage{amsmath}
\usepackage{amssymb}
\usepackage[applemac]{inputenc}
\usepackage{graphicx}
\usepackage{color}
\newcommand{\EE}{\mathbb{E}}

\newcommand{\PP}{ \mathbb{P}}
\renewcommand{\P}{ \mathcal{P}}

\newcommand{\QQ}{ \mathbb{Q}}

\newcommand{\R}{{\mathcal R}}

\newcommand{\E}{{\mathcal E}}
\newcommand{\RR}{{\mathbb R}}
\newcommand{\NN}{{\mathbb N}}
\newcommand{\F}{\mathcal F}
\newcommand{\Q}{\mathcal Q}
\newcommand{\C}{\mathcal C}
\newcommand{\dd}{\mathrm d}
\renewcommand{\L}{\mathcal L}

\newcommand{\essinf}{\operatorname*{\mathrm{ess\,inf}}}

\newcommand{\supp}{\mathrm{supp}\;}

\newenvironment{proof}[1][Proof]{\textbf{#1.} }{\ \rule{0.5em}{0.5em}}
\newtheorem{remark}{\textbf{Remark}}[section]
\newtheorem{lemma}{\textbf{Lemma}}[section]
\newtheorem{theorem}{\textbf{Theorem}}[section]
\newtheorem{corollary}{\textbf{Corollary}}[section]
\newtheorem{proposition}{\textbf{Proposition}}[section]

\numberwithin{equation}{section}
\title{ Generalized entropy minimization under full marginal constraints \author{
Julio Backhoff-Veraguas  \thanks{ University of Twente, Department of Applied Mathematics, Enschede, The Netherlands.} \and     Joaqu\'in Fontbona  \thanks{ Center for Mathematical Modeling
(UMI-CNRS 2807), University of Chile,  Beauchef 851, Santiago, Chile. Support by Conicyt Basal-CMM Proyecto/Grant PAI AFB-170001 and Programa Iniciativa Cient\'ifica Milenio
grant number NC120062 through Nucleus Millenium Stochastic Models of Complex and Disordered
Systems is gratefully acknowledged.}  } }

\begin{document}
\maketitle
\abstract{ 
We consider the problem of minimizing a generalized relative entropy, with respect to a reference diffusion law, over the set of path-measures with fully prescribed marginal distributions. When dealing with the actual relative entropy, problems of this kind have appeared in the stochastic mechanics literature, and minimizers go under the name of Nelson Processes.
 
 Through convex duality and stochastic control techniques, we obtain in our main result the full characterization of minimizers, containing the related results in the pioneering works of Cattiaux \& L\'eonard \cite{CattLeo} and Mikami \cite{Mikami_FP} as particular cases. We also establish that minimizers need not be Markovian in general, and may depend on the form of the generalized relative entropy if the state space has dimension greater or equal than two. Finally, we illustrate how generalized relative entropy minimization problems of this kind may prove useful beyond stochastic mechanics, by means of two applications: the analysis of certain mean-field games, and the study of scaling limits for a class of backwards SDEs.

}
\medskip\medskip\medskip


\normalsize

\section{Introduction}
\label{intro}

\subsection*{Overview}Let $t\mapsto \mu_t$ be a given weakly continuous flow of probability measures on $\RR^q$. In this work we consider the following variational problem:
\begin{equation}\label{p aux}
\inf \left\{ \EE^{\QQ}\left [ \int_0^{ T} g^*\left ( t,X_t,\sigma'(t,X_t) \beta_t\right )\dd t \right ]:\,\QQ\circ X_t^{-1}=\mu_t,\,\,\forall t\in [0,T]  \right\},
\end{equation}
where the optimization is performed over all probability measures $\QQ$ solution of the martingale problem with coefficients $(b+a\beta,a)$, where $a=\sigma\sigma'$ and $b$ are fixed functions (contrary to $\beta$). When $g^*(t,x,\cdot)=\|\cdot\|^2/2$ this corresponds to minimizing the relative entropy of $\QQ$ with respect to the law of the solution of the martingale problem $(b,a)$, given the flow of marginals constraint. In such case a unique extremal solution to \eqref{p aux} is known to exist provided this problem is finite, and it is known to be a Markovian measure. The construction of such trajectorial law goes under the name of ``Nelson Processes'' in the stochastic mechanics literature; see \cite{CattLeo_AIHP,CattLeo} and references therein. For generalized entropy minimization as in \eqref{p aux}, the problem has only been analysed in \cite{Mikami_FP}, to the best of our knowledge. Our aim is to:

\begin{itemize}
\item Obtain existence, duality, and characterization of the optimizers of \eqref{p aux}.
\item Establish the nature of the optimizer of \eqref{p aux} in terms of its Markovianity and robustness (i.e.\ interplay between $g^*$, its growth, and the spatial dimension $q$).
\item Introduce novel applications for \eqref{p aux} beyond stochastic mechanics. 
\end{itemize}

\medskip

\subsection*{Proper Setting and Assumptions}\label{setting}

Let $$b:[0,T]\times\RR^q\to \RR^q,\,\,\,\, \sigma:[0,T]\times\RR^q\to \RR^{q\times q}.$$
%
Throughout we use the apostrophe ($'$) to denote transposition, and we let $$a:= \sigma\sigma',$$ which is then an $\RR^{q\times q}$-valued function. We work under the assumption
%
\begin{itemize}
\item[{\bf(A)}] $a$ is bounded and $b,\sigma$ are once differentiable in time, twice differentiable in space, and satisfy the usual linear-growth and Lipschitz conditions of It\^o theory. The matrix $\sigma$ is invertible. 
\end{itemize}
Let us define the differential operators:
\begin{align*}
\L = b'\nabla_x+\frac{1}{2}\sum_{i,j}a^{i,j}\partial^2_{x_i,x_j}\,\,, \hspace{20pt}
\L_t = \partial_t+\L.
\end{align*}
Under the above assumption, the martingale problem with generator $\L$ (one also says, with coefficients $(b,a)$) and domain $C_0^\infty((0,T)\times \RR^q)$ admits for each starting point $x$ a unique solution, which we denote $\PP_x$. Equivalently, the diffusion SDE 
\begin{equation}\dd X_t = b(t,X_t)\dd t+ \sigma (t,X_t)\dd W_t,\,\,\,\, X_0=x,\label{SDE}
\end{equation}
has a unique weak solution. 

From now on we fix a starting distribution $m_0(\dd x)$ and define $$\PP := \int m_0(\dd x)\PP_x.$$ We also denote throughout by $X$ the canonical process on $\Omega:=C([0,T];\RR^q)$, and by $\{\F_t\}$ the canonical filtration. 
Further, we define
$$M_t:= X_t-X_0-\int_0^t b(s,X_s)ds = \int_0^t \sigma(s,X_s)\dd W_s \,.$$
%
%
For simplicity we write $\EE$ for expectation under $\PP$. With $\mu:=\{\mu_t\}_t$ we also denote
$$
\Q(\mu):=\left\{ \QQ\in\P(\Omega):\,\,
\begin{array}{c}
 \QQ\circ X_t^{-1}=\mu_t \,\,\forall t\in [0,T],\,\, M_{\cdot}-\int_0^{\cdot} a(t,X_t) \beta^\QQ_t\dd t \\  \mbox{ is a $\QQ$-martinglale and }\langle M\rangle_\cdot=\int_0^\cdot a(t,X_t)\dd t
\end{array}
\right \}.
$$
\noindent We stress that $\Q(\mu)$ may contain measures singular with respect to $\PP$. We can now properly define \eqref{p aux}: the \emph{primal} optimization problem central to this article is:
%
%
%
\begin{align}
\inf_{\QQ\in \Q(\mu)} \EE^{\QQ}\left [ \int_0^{ T} g^*\left ( t,X_t,\sigma'(t,X_t)\beta^{\QQ}_t\right )\dd t \right ], \label{eqprimal extendido}\tag{$P^{ext}[\mu]$}
\end{align}
where $$g^*:[0,T]\times \RR^q\times \RR^q\to \RR_+\, .$$
It is implicitly assumed that $\beta^\QQ$ above is a predictable functional (of $X$) s.t.\ the $\dd t$-integral is well-defined. 

Regarding the family $\mu:=\{\mu_t\}_t$ we make the standing assumption:

\begin{itemize}
\item[{\bf (B)}] Each $\mu_t$ is a Borel probability measure on $\RR^q$ and the function $t\mapsto \mu_t$ is continuous w.r.t.\ the usual weak topology of measures on the target space. We further assume $$\mu_0=m_0.$$
\end{itemize}
%

We let $g$ be the convex conjugate of $g^*$ w.r.t.\ the last argument:
$$(t,x,y)\in[0,T]\times\RR^q\times\RR^q\mapsto g(t,x,y):=\sup_{z\in\RR^q}\{\langle z,y \rangle - g^*(t,x,z)\}.$$ 
Out standing assumption on $g^*$ is

\begin{itemize}
\item[{\bf (C)}] $g^*$ is measurable in the first two coordinates, whereas it is strictly convex, even and continuously differentiable 
in the last one. Moreover, we have
\begin{enumerate}
\item $g^*(t,x,y)=0\iff y=0,$
\item $(t,x)\mapsto \sup_{|y|\leq n} g^*(t,x,y)$ is $\mu_t(\dd x)\dd t$-integrable for each $n\in\NN$,
\item $ \limsup_{|y|\to 0}\frac{g^*(t,x,y)}{|y|}=0$\mbox{  and for some $p>1$ we have uniformly on $(t,x)$}$$\liminf_{|y|\to \infty}\frac{g^*(t,x,y)}{|y|^p}>0,$$ 
\item $\exists C> 1$ and $h:[0,T]\times\RR^q\to\RR_+$ $\mu_t(\dd x)\dd t$-integrable s.t. $$g^*(t,x,2y)\leq Cg^*(t,x,y)+ h(t,x),$$
\item $\exists \ell>1$ and $H:[0,T]\times\RR^q\to\RR_+$ $\mu_t(\dd x)\dd t$-integrable s.t. $$g^*(t,x,y)\leq \frac{g^*(t,x,\ell y)}{2\ell}+ H(t,x).$$
\end{enumerate}
\end{itemize}
%
%
Note that in particular $g^*$ is jointly measurable, and necessarily $g$ is non-negative and finite-valued. Furthermore, $g$ is strictly convex, differentiable and even w.r.t.\ the last coordinate.  

\begin{remark}
The case $g^*(t,x,y)=\|y\|^2$ corresponds to the entropy criterion. Notice that $g^*(t,x,y)=R(t,x)\|y\|^p$, with $1<p<\infty$ and $R(\cdot,\cdot)$ integrable and uniformly strictly positive, satisfies the above assumptions.  More generally, $g^*(t,x,y)=R(t,x)\|y\|^p[\, 1+|\,\log\|y\|\,|\,]$ does it too, and so forth.
\end{remark}

We shall occasionally refer to the property $$\liminf_{|y|\to \infty}\frac{g^*(t,x,y)}{|y|^2}>0, \mbox{uniformly on $(t,x)$},$$ by saying that  ``\emph{$g^*$ has at least quadratic growth}.'' This is \emph{not} assumed for most results in this article.

\subsection*{Main results}\label{main results}

We introduce the space of test functions
$$\C:= \left \{ w:[0,T]\times\RR^q\to \RR\,\,\in C^{1,2}: \supp(w)\subset (0,T)\times\RR^q  \mbox{ compact}  \right \}  , $$
%
with an associated variational problem:
\begin{equation}\tag{$D_0[\mu]$}
\label{eq dual original}
\sup_{w\in \C}\, \iint  \left\{ \L_t w(t,x) - g(t,x,\sigma'(t,x)\nabla w(t,x))\right \}\mu_t(\dd x) \dd t.
\end{equation}
%

Problem \eqref{eq dual original} has to be supplemented with a suitable extension, namely
\begin{equation}\tag{$D[\mu]$}
\label{eq dual extendido}
\sup_{\psi\in L^g_\nabla}\, L(\psi)- \iint   g(t,x,\sigma'(t,x)\psi(t,x))\mu_t(\dd x) \dd t.
\end{equation}
We postpone the definition of $L^g_\nabla$ and the interpretation of the linear functional $L(\psi)$ to Section \ref{duality}. 
Problems \eqref{eq dual original} and \eqref{eq dual extendido} are referred to as the \emph{dual} and the \emph{extended dual} problems respectively. We can now state the main structural result of the article. 


\begin{theorem}\label{Thm Main}
There is no duality gap:
\begin{align}\label{eq values equal} 
\mbox{value}\eqref{eqprimal extendido}\,\, =\,\, \mbox{value}\eqref{eq dual original}\,\, =\,\, \mbox{value}\eqref{eq dual extendido}. 
\end{align}
If this common value is finite, then the primal problem is attained by a unique $\QQ\in\Q(\mu)$, and 
the extended dual problem is attained by a $\mu_t(\dd x)\dd t$-a.s.\ unique ${\Psi}\in L_\nabla^g$.  These optimizers are related as follows: 
Under $\QQ$ the canonical process satisfies
\begin{align}\label{formula crucial}
\dd X_t =\left\{ b(t,X_t)- \sigma(t,X_t)\nabla g(t,X_t,\sigma'(t,X_t)\Psi(t,X_t))\right \}\dd t + \sigma(t,X_t)\dd W_t\, ,
\end{align}
and the common value in \eqref{eq values equal} equals 
$$\iint g^*(t, x,\, \nabla g(t ,x,\sigma'(t,x)\Psi(t,x) )\,)     \mu_t(\dd x) \dd t. $$
%
If furthermore $g^*$ has at least quadratic growth, then $\QQ\ll\PP$ and 
\begin{equation}\label{formula crucial abs}
\frac{\dd \QQ}{\dd \PP}= \,\E\left ( -\int \nabla g(t,X_t,\sigma'(t,X_t){\Psi}(t,X_t))'\sigma^{-1}(t,X_t)\dd M_t   \right )_T\,.
\end{equation}

\end{theorem}

We now provide two applications of the main result. First we ask whether the optimal measure for the primal problem has the Markov property. Recall that this does not simply follow from the coefficients being ``Markovian,'' and indeed we will show that the Markov property may fail for the optimal measure. This answers an open question in \cite{Mikami_FP} to the negative.

\begin{corollary}\label{coro mark}
There is $\PP$ and $\mu$ with $value\eqref{eqprimal extendido}<\infty$, for which the optimal solution $\QQ$ does not have the Markov property. On the other hand, if we assume that \eqref{eqprimal extendido} is attained by a probability measure absolutely continuous w.r.t.\ $\PP$ (which is guaranteed if $g^*$ has at least quadratic growth and the problem is finite),
%
then the optimal $\QQ$ must have the Markov property. 
\end{corollary}


Second, we address the following question: is the optimizer of the primal problem \emph{universal}, i.e.\ independent of the concrete $g^*$?. Our insight is that the answer depends on the dimension $q$:

\begin{corollary}\label{coro univ}
In dimension one ($q=1$) the solution of the primal problem does not depend on the cost $g^*$, as long as {\bf (C)} is fulfilled. In higher dimensions ($q\geq 2$) there is dependence on $g^*$. 
\end{corollary}

The fact that the optimizer is universal for dimension one, and that otherwise the optimizer does depend on the cost criterion, attests to the richness of the problem. 

\subsection*{Comparison with the literature}

Problem \eqref{eqprimal extendido} was first analyzed in Mikami's \cite{Mikami_FP}. Unlike in that article, we treat the subject directly, rather than as a limiting problem where only finitely many marginals are prescribed. This is the main methodological difference between the two works. In particular, this allows us to obtain duality directly with a continuum of prescribed marginals. The emphasis on duality theory allows us to relax the requirements on the cost function $g^*$, which in \cite{Mikami_FP} is assumed to be rather smooth owing to the use of PDE theory (strong solutions thereof). We also cover the case where $\PP$ is a diffusion law, rather than just Wiener measure; in particular, we make no use of uniform ellipticity. Other important differences are: the treatment of applications outside of the realm of stochastic mechanics (they will be given in Section \ref{sec applications} below), and a detailed study of the universality and Markovianity of the optimal primal solutions. In this last regard, we answer an open question in \cite{Mikami_FP} to the negative. 

Our duality approach is closest to Cattiaux \& L\'eonard's \cite{CattLeo}, where the entropic case is dealt with. Unlike in that article however, we do not use large deviations arguments but only duality and stochastic control techniques, and we cover generalized entropies rather than the relative entropy only. We also make use of backwards SDE techniques as in the works of Drapeau, Kupper, Tangpi and others \cite{minsupersol,dualrepminsupersol,markovminsupersol}.

A number of well studied problems in the literature share a similar nature with Problem \eqref{eqprimal extendido}. For instance in the works on \emph{Markovian projections} of Semimartingales by Bentata, Brunick, Cont, Gy\"ongy, Shreve \cite{Gyongy-mimicking,Gyongy-mimicking2,Shreve-mimicking,BeCo15} among others. On a similar note, this is close to the so-called \emph{Peacock problem} explored by  Kellerer \cite{Ke73}, Lowther \cite{Lo08b}, Hirsch \& Profeta \& Yor \cite{HiPr11}, Beiglb\"ock \& Huesmann \& Stebbeg \cite{BeHuSt15}, Juillet \cite{Ju16}, K\"allblad \& Tan \& Touzi \cite{KaTaTo17}, and many other authors: given a continuum of marginals in increasing convex order, does there exist a \emph{simple} martingale (eg.\ Markovian) having them as marginals? Another close cousin of Problem \eqref{eqprimal extendido} is the celebrated \emph{Schr\"odinger problem} (also called entropic optimal transport), wherein only initial and final marginal distributions are prescribed: we refer to the survey by L\'eonard \cite{Le14} for a detailed historical account and to the works by Backhoff, Benamou, Carlier, Chen, Confroti, Cuturi, Gentil, Georgiou, L\'eonard, Nenna, Pammer, Pavon, Peyr\'e \cite{Cu13,BeCaCuNePe15,ChGePa16,Co19,BaCoGeLe19,Co19,BaPa20} for a sample of recent developments. By mixing the Scrh\"odinger problem with \eqref{eqprimal extendido} in the entropic case, one obtains the so-call Bredinger Problem, which can be seen as a regularized version of Brenier's incompressible fluid model \cite{Br89,Br93}; See the works by Arnaudon, Baradat, Benamou, Carlier, Cruzeiro, L\'eonard, Monsaingeon, Nenna, Zambrini \cite{ArCrLeZa17,BaMo19,BeCaNe19,Ba20}.

\subsection*{Outline} First we provide in Section \ref{sec applications} applications for the results hitherto obtained, namely for Mean-Field games and non-exponential large deviations of empirical flows. The rest of the article is devoted to the proofs of the main result and its corollaries. In Section \ref{sec primals} we look in depth at the primal problem. In Section \ref{duality} we introduce the dual problem(s). In Section \ref{No gap} we establish the absence of duality gap. In Section \ref{proof main result} we prove the main theorem. Finally in Section \ref{sec: ex} we provide important (counter)examples and complete the proofs of the main corollaries.

\section{Applications}\label{sec applications}
So far we have worked with a fixed flow of marginals $\mu$, in this part we shall let $\mu$ vary. The notation so far has been set up to deal with this situation.

\subsection{McKean-Vlasov control and Mean-Field games of potential type}

Let us write 
\begin{equation}
\Q: = \bigcup_{\mu\text{ satisfying \bf{(B)}}} \Q(\mu).\label{Q curly}
\end{equation}
We consider the following McKean-Vlasov control problem in canonical space (i.e.\ in weak formulation):
\begin{align}
\label{MKV}\tag{$MKV_0$}
\inf\left\{ \EE^{\QQ}\left[\int_0^T g^*(t,X_t,\sigma'(t,X_t)\beta^\QQ_t(X))  \dd t\,\right ] + \int_0^T R_t[\QQ\circ X_t^{-1}]\dd t\, :\,\,  \QQ\in \Q\,  \right\}.
\end{align}
Here $$(t,m)\ni [0,T]\times \mathcal{P}(\RR^q)\mapsto R_t[m]\in\RR_+,$$ is assumed measurable. We have the following technical result whose straightforward proof we omit.

\begin{lemma}\label{lem MKV}
Problem \eqref{MKV} is equivalent to
\begin{align}\label{MKV alt}\tag{MKV}
\inf\left\{ \text{value}\eqref{eqprimal extendido} + \int_0^T R_t[\mu_t]\dd t\, :\, \mu=\{\mu_t\}_t \text{ satisfies Assumption }\bf{(B)} \right\}.
\end{align}
In particular: $\mu$ is an optimizer for \eqref{MKV alt} and $\QQ$ is an optimizer for \eqref{eqprimal extendido} iff $\QQ$ is an optimizer for \eqref{MKV} and the marginals of $X$ under $\QQ$ are given by $\mu$.
\end{lemma}


The goal of this part of the article is to illustrate the use of Theorem \ref{Thm Main} to obtain that the ``optimal control'' $\beta$ is of Markovian feedback form. The same will be true for associated Mean Field games that we will introduce shortly. We stress that this is then a purely variational argument for the existence of optimal Markov controls, as opposed to analytical arguments. We refer to \cite{buckdahn2009mean,buckdahn2011general,andersson2011maximum,meyer2012mean,BFY,CD_AP,CD_AP,AcBaCa19} for references on McKean-Vlasov control (also known as mean-field control), to \cite{CarmonaDelarue_book_vol_I} for extensive references on mean-field games, to the works of Lacker \cite{La15,La17} for the general question of existence of Markovian optimizers, and to \cite{AcBaJi20} for dynamic potential games. We make all simplifying assumptions necessary to keep technicalities at a minimum.

\begin{proposition}\label{prop MKV}
Assume that \eqref{MKV} is finite, and that: $$\text{for all }t\,:\,\,\, m\mapsto R_t[m] \text{ is lower-semicontinuous}.$$ Then Problem \eqref{MKV} has an optimizer $\QQ$ for which the associated optimal control is Markov: $\beta^{\QQ}_t(X)=\beta^{\QQ}_t(X_t)$ for each $t$.
\end{proposition}

\begin{proof}
Let $\QQ^n$ be $(1/n)$-optimizers for \eqref{MKV}. It follows that $$\EE^{\QQ^n}\left[\int_0^T g^*(t,X_t,\sigma'(t,X_t)\beta^{\QQ^n}_t(X))  \dd t\,\right ]\leq 1+ value\eqref{MKV}.$$ 
By Lemma \ref{lem folk}, $\{\QQ^n\}_n$ is tight. We denote by $\QQ$ an accumulation point. Again by this lemma we deduce $\QQ\in\Q$. Analogously, and due to the assumption on $R$ (plus Fatou's lemma), we derive  the lower-semicontinuity of the objective function. This implies the optimality of $\QQ$ for \eqref{MKV}. Denoting $\mu$ the flow of marginals of this measure, we clearly have that $\mu\text{ satisfyies \bf{(B)}}$, and necessarily $\QQ$ is optimal for \eqref{eqprimal extendido}.  By Theorem \ref{Thm Main}, the associated $\beta^{\QQ}$ is of the desired form.
\end{proof}\\

From now on we assume that $R$ is differentiable, meaning that the following directional derivatives exist 
$$\lim_{\epsilon\to 0+}\frac{R_t[m+\epsilon(\bar{m}-m)]-R_t[m]}{\epsilon} = \int \nabla R_t[m](y)(\bar{m}-m)(\dd y),$$
along with a bounded measurable function $\nabla R_t[m]:\RR^q\to\RR$. 

We consider the following Mean Field game (MFG) of potential form on canonical space (this is again a weak formulation):  Find $(\QQ,\mu)$ such that
\begin{enumerate}
\item[(1)] $\QQ$ attains $$\inf\left \{ \EE^{\tilde{\QQ}}\left[ \int_0^T \left ( g^*(t,X_t,\sigma'(t,X_t)\beta^{\tilde{\QQ}}_t(X))+\nabla R_t[\mu_t](X_t) \right ) \dd t\, \right ] \,:\, \tilde{\QQ}\in \Q \right\},$$
\item[(2)] $\QQ\circ X_t^{-1}=\mu_t$ for all $t\in [0,T]$.
\end{enumerate}

Leveraging on Proposition \ref{prop MKV}, we prove the existence of a Mean Field equilibrium where the optimal control $\beta^{\QQ}$ is Markovian.

\begin{proposition}\label{prop MFG}
Under the conditions in Proposition \ref{prop MKV}, and the differentiability assumption on $R$, let $\QQ$ be any optimizer for \eqref{MKV} and $\mu_t:=\QQ\circ X_t^{-1}$. Then $(\QQ,\mu)$ is a solution (i.e.\ an equilibrium) to the Mean Field game $(1)$-$(2)$ above, and the associated control $\beta^{\QQ}$ is Markov. 
\end{proposition}

\begin{proof}
Let $\QQ$ as in Proposition \ref{prop MKV}, with marginals $\mu$. By Lemma \ref{lem MKV} we have 
\begin{align*}& value(P^{ext}[\mu])+\int_0^T R_t[\mu_t]\dd t \\ 
\leq & \, value(P^{ext}[\mu + \epsilon(\bar{\mu}-\mu)])+\int_0^T R_t[\mu_t + \epsilon(\bar{\mu}_t-\mu_t)]\dd t \\ 
\leq & \, \epsilon \, value(P^{ext}[\bar{\mu}])+(1-\epsilon)\,value(P^{ext}[\mu])+\int_0^T R_t[\mu_t + \epsilon(\bar{\mu}_t-\mu_t)] \dd t.
\end{align*}
Indeed, one can see the convexity of $value(P^{ext}[\cdot])$ either directly or as a consequence of the absence of duality gap (Theorem \ref{Thm Main}) since the dual problem $value(D[\cdot])$ is obviously convex. Rearranging we obtain
$$value(P^{ext}[\mu])\leq value(P^{ext}[\bar{\mu}])+\lim_{\epsilon\to 0+}\,\int_0^T \frac{ R_t[\mu_t + \epsilon(\bar{\mu}_t-\mu_t)] - R_t[\mu_t]}{\epsilon}\dd t.$$
By dominated convergence and the differentiability assumption, we deduce
$$\text{value}(P^{ext}[\mu])\leq \text{value}(P^{ext}[\bar{\mu}])+\int_0^T \int \nabla R_t[\mu_t](y)(\bar{\mu}_t-\mu_t)(\dd y)\,\dd t,$$
so
\begin{align*}& value(P^{ext}[\mu])+ \int_0^T\int  \nabla R_t[\mu_t](y) \dd\mu_t(y)\dd t  \\ \leq &  value(P^{ext}[\bar{\mu}])+ \int_0^T\int  \nabla R_t[\mu_t](y) \dd\bar{\mu}_t(y)\dd t . \end{align*}
Since $\bar{\mu}$ is arbitrary, this is clearly equivalent to saying that $(\QQ,\mu)$ is a Mean Field game equilibrium.
\end{proof}

\subsection{A generalized Laplace principle for empirical flow of particles}

We interpret here the value of our primal problem \eqref{eqprimal extendido}, seen as a function of the flow $\mu$, as the rate function of a non-exponential Laplace principle for empirical flow of marginals. In this way we come full circle with the work \cite{CattLeo}, where the authors start from an exponential Laplace principle, and then study \eqref{eqprimal extendido} in the entropic case. Indeed, we do the opposite here, starting from the study of \eqref{eqprimal extendido} and then referring to a non-exponential Laplace principle. Furthermore, we cover situations vastly more general than the entropic case. Our starting point is the work \cite{Lacker16} by Lacker, and its Wiener space specialization \cite{BaLaTa18} by Lacker, Tangpi, and one of the authors. We let $\gamma$ denote the Wiener measure in state space $\mathbb R^q$ and started at the origin, and assume for simplicity that $$g^*(t,x,y)=g^*(y),$$
and that $m_0$ is concentrated on a point (w.l.o.g.\ the origin). We have
\begin{proposition}
Let $F$ be a real-valued, measurable and bounded functional over flows of probability measures, namely $F\in B_b(\, C([0,T];\P(\RR^q))\,)$. Let $\{W^i\}_{i\in\NN}$ distributed like $\gamma^{\otimes\NN}$ and $\{X^i\}_{i\in\NN}$ be the associated i.i.d.\ sequence of solutions to \eqref{SDE}. We consider the following backwards SDE under $\gamma^{\otimes\NN}$:
$$\dd Y^n_t = -g(\sqrt{n}Z_t)\dd t + Z_t \dd W^{(n)}_t\,\,\, ,\,\,\, Y_T^n=nF\left(t\mapsto \frac{1}{n}\sum_{i\leq n}\delta_{X^i_t}\right),$$
where $W^{(n)} $ is the $\gamma^{\otimes\NN}$-Brownian motion obtained by appropriate scaling and consecutive concatenation of $W^1,\dots,W^n$ over the time-index set $[0,T]$. Then
\begin{align*}
\lim_{n\to\infty}\frac{1}{n}Y_0^n &= \sup_\mu\left\{ F\left(t\mapsto \mu_t \right) - \inf_{\QQ\in\Q(\mu)} E^\QQ\left[\int_0^T g^*(\sigma'(t,X_t)\beta^\QQ_t)\dd t\right]\right\} \\ &= \sup_\mu\left\{ \,F\left(t\mapsto \mu_t \right) - \mbox{value}\eqref{eqprimal extendido}\,\right\}.
\end{align*}
\end{proposition}

\begin{proof}

{\bf Step 1}: We recall here the crucial result of \cite{BaLaTa18}. 
Let $\tilde{F}\in B_b(\,\P(\,C( [0,T];\RR^q)\,)\,)$. With the same ingredients as in the statement, we have
$$\lim_{n\to\infty} \frac{1}{n}\tilde{Y}_0^n = \sup\left\{\tilde{F}(\tilde{\QQ}) -E^{\tilde{\QQ}}\left[\int_0^T g^*(q^{\tilde{\QQ}}_t)\dd t\right]\,:\,\tilde{\QQ}\,\,s.t.\,\, X_\cdot-\int_0^\cdot q^{\tilde{\QQ}}_t\dd t \text{ is $\tilde{\QQ}$-B.m.} \right\},$$
where $\tilde{Y}^n$ solves the same BSDE as $Y^n$ but with the terminal condition 
$$\tilde{Y}^n_T= n\tilde{F}\left( \frac{1}{n}\sum_{i\leq n}\delta_{W^i} \right ),$$
under $\gamma^{\otimes\NN}$.

{\bf Step 2}: We now move from Wiener measure to the diffusion law $\PP$. Since \eqref{SDE} has a unique strong solution, there is a measurable map $H$ between path-spaces such that $X=H(W)$. For $\hat{F}\in B_b(\,\P(\,C( [0,T];\RR^q)\,)\,)$, we consider $\tilde{F}(\QQ):=\hat{F}(\QQ\circ H^{-1})$. Observe that pointwise $$\tilde{F}\left( \frac{1}{n}\sum_{i\leq n}\delta_{\omega^i} \right ) = \hat{F}\left( \frac{1}{n}\sum_{i\leq n}\delta_{H(\omega^i)} \right ).
$$ 
Notice that $\tilde{\QQ}$ is associated to $q^{\tilde{\QQ}}_t$ iff $\QQ=\tilde{\QQ}\circ H^{-1}$ is associated to $\sigma'(t,X_t)\beta^{\QQ}_t = q_t^{\tilde{\QQ}}$.
This and Step 1 show that
$$\lim_{n\to\infty} \frac{1}{n}\hat{Y}_0^n = \sup\left\{\hat{F}({\QQ}) -E^{{\QQ}}\left[\int_0^T g^*(\sigma'(t,X_t)\beta^{\QQ}_t)\dd t\right]\,:\, \QQ \in\Q \right\},$$
where $\hat{Y}^n$ solves the same BSDE as $Y^n$ but with the terminal condition 
$$\tilde{Y}^n_T= n\hat{F}\left( \frac{1}{n}\sum_{i\leq n}\delta_{X^i} \right ),$$
under $\gamma^{\otimes\NN}$.

{\bf Step 3}: We now change the state space from $\P(\,C( [0,T];\RR^q)\,)$ to $C([0,T];\P(\RR^q))$, much as in the contraction principle in large deviations theory. Let $F$ as in the statement. Then $F$ can be seen as belonging to $B_b(\,\P(\,C( [0,T];\RR^q)\,)\,)$ via the identification
$$\tilde{F}(\QQ)=F(t\mapsto \QQ\circ X_t^{-1}).$$
Applying Step 2 to this $\tilde{F}$ we easily obtain the desired result and finish the proof.
\end{proof}\\

In the entropic case (i.e.\ when $g^*$ is quadratic), this Laplace principle is equivalent to a large deviations principle (LDP) for the same objects. It is unclear whether the above general result can be translated into a LDP of sorts. Nevertheless, we think it is a curious observation that generalized entropy minimization is so closely related to scaling limits of backwards SDEs.

\section{The primal problem}\label{sec primals}


Recall the notation $\Q$ from \eqref{Q curly}. Let 
\begin{align}\label{primalobj ext}
\tilde{I}(\QQ):=\EE^{\QQ}\left [ \int_0^T g^*\left ( t,X_t,\sigma'(t,X_t)\beta_t^{\QQ}\right )\dd t \right ],
\end{align}
 if  $\QQ\in\Q$, and otherwise we set $\tilde{I}(\QQ)=+\infty $. This is our primal objective function.
 
 \begin{lemma}\label{lem folk}
 The function $\tilde{I}$ is strictly convex, lower-semicontinuous with respect to weak convergence, and has tight sub-level sets (i.e.\ $\tilde{I}$ is inf-compact).
 \end{lemma}
 \begin{proof}
 This is folklore. It readily follows e.g.\ from \cite[Theorem 8.3]{BaPa20}.
 \end{proof}
 
 We now  prove that \eqref{eqprimal extendido} is attained.

\begin{lemma}\label{lem tightness -}
If $\text{value}\eqref{eqprimal extendido}<\infty$, this problem has a unique optimizer.
\end{lemma}

\begin{proof}
Immediate from Lemma \ref{lem folk} and the fact that the constraints are closed w.r.t.\ weak convergence.
\end{proof}

We will need further properties of the functional $\tilde{I}$ when we establish the absence of duality gap in Section \ref{No gap}. First we must introduce some terminology from stochastic analysis. We follow \cite{dualrepminsupersol}, in the simpler so-called translation-invariant setting. By a supersolution of a Backward Stochastic Differential Equation (BSDE) with generator $g$ and terminal condition $F(X)$ we mean a couple of processes $(Y,Z)$, the first one c\`adl\`ag adapted and the second predictable and making $\int Z\cdot\dd M$ a $\PP$-supermaringale, such that
\footnote{Strictly speaking, the stochastic integral term is often taken to be of the form $\int_s^tZ_r\dd W_r$ in the literature. Since our $\sigma$ is invertible we can and prefer to write $\int_s^tZ_r\dd M_r$, as $M$ is the most \emph{natural} martingale for us.}
$$\left \{
\begin{array}{rll}
Y_s - \int_s^t g(r,X_r,\sigma'(t,X_t)Z_r)\dd r + \int_s^t Z_r\cdot \dd M_r  & \geq \, Y_t  & \mbox{, for all }0\leq s\leq t\leq T,\\
Y_T & \geq \, F(X) & .

\end{array}
\right .
$$
\noindent Obesrve that $Y_0$ is $X_0$-measurable. A supersolution $(\bar{Y},\bar{Z})$ is said minimal if a.s. $\bar{Y}_t\leq Y_t$ for every $t$ and every supersolution $(Y,Z)$. Let us denote by ${\cal A}(F)$ the set of supersolutions. From our assumptions follows that $g(r,X_r,0)=0$, so if $F$ is essentially bounded we have that $(\|F\|_{\infty},0)\in{\cal A}(F)$. As proved originally in \cite{minsupersol}, and extended in \cite[Theorem 2.1]{dualrepminsupersol}, we may define the minimal supersolution operator by $E^g_t(F)=+\infty$ for all $t$ if ${\cal A}(F)=\emptyset$, and otherwise
$$
E^g_t(F):= \essinf \left \{   Y_t: (Y,Z )\in {\cal A}(F)\right \},
$$
in which case the process $E^g(F)$ is the minimal supersolution for the terminal condition $F$. Again, $E^g_0(F)$ is $X_0$-measurable. When $g$ has at most quadratic growth in its last component then $E^g(F)$ may reduce to the solution of the BSDE with generator $g$. In general, a BSDE may have no solutions (see \cite{superquadraticBSDE}) and this is the reason one works with supersolutions. 

\begin{lemma}
\label{identificacionconjugadas} 
Define 
\begin{equation}
\label{primalobj}
\QQ\ll\PP\mapsto I(\QQ)\,\, := \,\, \EE^{\QQ}\left [ \int_0^T g^*\left ( t,X_t,\sigma'(t,X_t)\beta_t^{\QQ}\right )\dd t \right ] .
\end{equation}
if $\QQ\circ X_0^{-1}=m_0$, and $+\infty$ otherwise.
The minimal supersolution operator (at time zero) is related to $I$ via the following conjugate relationship:
$$\int [E^g_0(F)](x_0) \dd m_0( x_0)\,=\, \sup_{\QQ\ll\PP} \{ \EE^\QQ[F] - I(\QQ)\},\,\, F\in L^{\infty}(\PP).$$
The converse is also true, namely
\begin{align}\label{dual rep I}
I(\QQ)\,=\, \sup_{F\in L^\infty(\PP)}\left\{ \EE^\QQ[F]-\int [E^g_0(F)](x_0) \dd m_0( x_0 )\right\},\,\, \QQ\ll\PP.
\end{align}
\end{lemma}

\begin{proof}
By regular disintegration of $\QQ$ w.r.t.\ its initial condition, and the fact that the space $L^{\infty}$ is decomposable, it is elementary to see that proving  the conjugate duality relations in this lemma can be reduced to the case when $m_0$ is concentrated in a singleton. We now assume this. Then the first statement is \cite[Theorem 3.4]{dualrepminsupersol}, upon observing that what the authors call $q$ is our $\beta$ and that there is no ``discounting factor'' in our case since our $E^g$ is translation-invariant. The proof of \eqref{dual rep I} can be found in  \cite[Theorem 3.10]{dualrepminsupersol}, more precisely in the part of the proof entitled \emph{Second equality} therein (again, there is no discounting factor for us), if we assume that $\QQ\sim\PP$. The case $\QQ\ll\PP$ is obtained by convexity and elementary computations. 
\end{proof}

\begin{lemma}\label{identificacionconjugadas ext}
We have $\tilde{I}(\QQ)\leq I(\QQ)$ with equality if $\QQ\ll\PP$. Accordingly,
\begin{align}\label{eq ineq}
\int [E^g_0(F)](x_0) \dd m_0( x_0)\,\leq\, \sup_{\QQ} \{ \EE^\QQ[F] - \tilde{I}(\QQ)\},
\end{align}
for $F$ Borel bounded. If $F$ is lower semicontinuous and bounded from below, then there is equality in \eqref{eq ineq}.

\end{lemma}

\begin{proof}
Given $\QQ$, if $\beta,\bar{\beta}$ satisfy the conditions on $\beta^\QQ$ for  \eqref{primalobj ext}, then $a(t,X_t)(\beta-\bar{\beta})(t,X)=0$ holds $\dd\QQ\times\dd t$-a.s.\ and from here $\sigma'(t,X_t)\beta(t,X) = \sigma'(t,X_t)\bar{\beta}(t,X)$ $\dd\QQ\times\dd t$-a.s. Ergo the value of $\tilde{I}$ is well-defined. If $\QQ$ is not abs.\ continuous then $\tilde{I}(\QQ)\leq I(\QQ)$ is trivial. Otherwise, we obtain by Girsanov that  $M_{\cdot}-\int_0^{\cdot}a(t,X_t)\beta^\QQ_t\dd t $ is a $\QQ$-martingale with quadratic variation process $\int_0^\cdot a(t,X_t)\dd t$, where $\dd\QQ/\dd\PP=\E(\int \beta^\QQ\dd M)$. So there is equality in that case. As for \eqref{eq ineq}, it follows from Lemma \ref{identificacionconjugadas}, whereas the equality case is contained in \cite{BaLaTa18}. 
\end{proof}

\section{The dual problem and relevant function spaces}\label{duality}

We start by motivating the relevance of \eqref{eq dual original}.

\begin{lemma}\label{lem weak duality}
Weak duality holds: 
$ \mbox{value}\eqref{eqprimal extendido}\,\, \geq\,\,\mbox{value} \eqref{eq dual original}$.
\end{lemma}
\begin{proof}By definition of convex conjugates, and since $g$ is even in the last argument, we have for any admissible $\QQ,w$ that
\begin{align*}
& \EE^{\QQ}\left [ \int_0^T g^*\left ( t,X_t,\sigma'(t,X_t)\beta^{\QQ}_t\right )\dd t \right ]\\ \geq& \EE^{\QQ}\left [ \int_0^T \{-(\beta_t^\QQ)' a(t,X_t)\nabla w(t,X_t) - g(t, X_t,\sigma'(t,X_t)\nabla w(t,X_t))\}\dd t \right ] \\
=& \EE^{\QQ}\left [ \int_0^T \{\L_t w(t,X_t) - g(t, X_t,\sigma'(t,X_t)\nabla w(t,X_t))\}\dd t \right ]\\
=& \iint [\L_t w(t,x)- g(t,x,\sigma'(t,x)\nabla w(t,x))]\mu_t(\dd x) \dd t.
\end{align*}
Indeed, since $\QQ$ is a solution to the martingale problem $(b+a\beta^\QQ,\, a)$, and as $w\in \C$ implies that $\nabla w$ is bounded, we have
%
$$\EE^\QQ\left[\int_0^T \left \{ (\beta^\QQ_t)'a(t,X_t) \nabla w(t,X_t) + \L_t w(t,X_t)\right\}\dd t\right]=0.$$
\end{proof}

Thus we are entitled to call \eqref{eq dual original} the \emph{dual} problem. We shall soon extend this problem, but first we need to introduce a few more elements. Let us define a semi-norm on functions $\psi:[0,T]\times\RR^q\to\RR^q$ as follows
$$\|\psi\|_g := \inf\left \{ \ell\geq 0:\iint g(t,x,\sigma'(t,x)\psi(t,x)/\ell)  \mu_t(\dd x)\dd t\leq 1\right\},$$
as well as the following Orlicz-like space:
\begin{align*}
L^g&:= \left\{ \psi:[0,T]\times\RR^q\to\RR^q\,\,\in L^1(\mu_t(\dd x)\dd t) : \|\psi\|_g<\infty\right\}.
\end{align*}
Under Assumption {\bf (C)} we actually have (see proof of Lemma \ref{lem Kozek} below)
$$L^g =\left\{ \psi:[0,T]\times\RR^q\to\RR^q : \forall \alpha >0,\,\, \iint g(t,x,\alpha\sigma'(t,x)\psi(t,x))  \mu_t(\dd x)\dd t<\infty \right\}. $$
We cannot call $L^g$ an actual Orlicz space because of the presence of the time-space parameters $(t,x)$ in $g$. It is however an Orlicz-Musielak space (see \cite{KozekBan,Kozekint}).
{
Similarly, we define
\begin{align*}
L^{g^*} &:= \left\{ \psi:[0,T]\times\RR^q\to\RR^q\,\,\in L^1(\mu_t(\dd x)\dd t) : \|\psi\|_{g^*}<\infty\right\}.\\
\|\psi\|_{g^*} &:= \inf\left \{ \ell\geq 0:\iint g^*(t,x,\sigma'(t,x)\psi(t,x)/\ell)  \mu_t(\dd x)\dd t\leq 1\right\} .
\end{align*} 

\begin{lemma}\label{lem Kozek}
Identifying $\mu_t(\dd x)\dd t$-a.s.\ equal functions, the semi-norm $\|\cdot\|_g$ (respect.\ $\|\cdot \|_{g^*}$) is an actual norm on $L^g$ (respect.\ $L^{g^*}$). The norm dual of $L^g{^*}$ is isometrically isomorphic to $L^g$, and both are reflexive Banach spaces. The duality pairing is
$$(\beta,\psi)\ni L^{g^*}\times L^g\mapsto \iint \beta(t,x)'a(t,x)\psi(t,x) \mu_t(\dd x)\dd t.$$
\end{lemma}

\begin{proof}
Observe that the convex conjugate of $g(t,x,\sigma'(t,x)\cdot)$ is $$h(t,x,\cdot):=g^*(t,x,\sigma(t,x)^{-1}\cdot).$$
Let us call $L^h$ the Orlicz-like space 
\begin{align*}
L^{h} &:= \left\{ \psi:[0,T]\times\RR^q\to\RR^q\,\,\in L^1(\mu_t(\dd x)\dd t) : \|\psi\|_{h}<\infty\right\}.\\
\|\psi\|_{h} &:= \textstyle\inf\left \{ \ell\geq 0:\iint {h}(t,x,\psi(t,x)/\ell)  \mu_t(\dd x)\dd t\leq 1\right\} .
\end{align*} 
Notice that \cite[Conditions A and B, p. 109-110]{KozekBan} are fulfilled. Indeed taking $F:=L^{h}$ in the author's notation, the first condition is a consequence of $F$ containing functions taking two values, whereas the second condition follows from Assumption ${\bf (C)}.2$. Also \cite[Definition 2.1.1, 2.1.2 and 2.1.3]{KozekBan} hold for $\Phi:=h$, thanks to Assumption ${\bf (C)}$. By \cite[Theorem 2.4]{KozekBan} $\|\cdot\|_{h}$ is a norm and $L^{h}$ is Banach. By Assumption ${\bf (C)}.4$ and \cite[Corollary 1.7.4]{Kozekint} we have that the norm dual of $L^h$ is isometrically isomorphic to $L^g$. In particular $\|\cdot\|_g$ is a norm and $L^g$ is Banach. Observe that Assumption ${\bf (C)}.5$ on $g^*$ implies that ${\bf (C)}.4$ holds but on $g$. Thus the equivalent expression for $L^g$ holds, and further applying \cite[Proposition 4.5]{KozekBan} and again \cite[Corollary 1.7.4]{Kozekint} we get that the norm dual of $L^g$ is isometrically isomorphic to $L^{h}$. Putting things together, this shows the reflexivity of both spaces. Now, the mapping
$$L^{g^*}\ni\phi \mapsto a\phi \in L^h,$$
is clearly an isometric isomorphism. It follows that we can identify $L^{g^*}$ and $L^h$, so the former is reflexive Banach and with dual isometrically isomorphic to $L^g$. Since the duality product between $L^g$ and $L^h$ is given by
$$\textstyle(\beta,\psi)\ni L^{h}\times L^g\mapsto \iint  \beta(t,x)'\psi(t,x) \mu_t(\dd x) \dd t,$$
we obtain the desired duality product between $L^{g^*}$ and $L^g$.
\end{proof}
}
\medskip 

Let us introduce the ``space of gradient fields'' 
\begin{align}\label{eqLnabla}
L^g_\nabla&:=  \overline{\{ \nabla w: w\in \C  \}}^{L^g}.
\end{align}

\begin{lemma}\label{lem extension functional}
Assume that  value$\eqref{eq dual original}<\infty$. Then there is a unique continuous linear functional
$$L: L^g_\nabla \to \RR,$$
for which $$\forall w\in\C:\,\, L (\nabla w)=\iint \L_t w(t,x)\mu_t(\dd x) \dd t.$$
%
\end{lemma}

\begin{proof}
From value$\eqref{eq dual original}<\infty$ we easily get
$$  
\left |L(\nabla w)\right | \leq \iint  g(t,x,\sigma'(t,x)\nabla w(t,x))\mu_t(\dd x) \dd t +\mbox{value}\eqref{eq dual original},
$$
for all $w\in\C$. Replacing $w$ by $w/\ell$, and choosing $\ell$ appropriately, we also get
\begin{equation} \label{eq mod continuity}
\left |L(\nabla w)\right | \leq \{1+ \mbox{value}\eqref{eq dual original}\} \,\,\|\nabla  w\|_g.
\end{equation}
By linearity and density of gradients in $L^g_\nabla$, we conclude.
\end{proof}

We denote by $(L_\nabla^g)^*$ the norm dual of $L_\nabla^g$ equipped with the $\|\cdot\|_{g}$-topology, i.e.
 $$(L_\nabla^g)^*:=\left\{ \ell:L_\nabla^g\to\RR\mbox{ linear and s.t. } \|\ell\|_{(L_\nabla^g)^*}:=\sup_{w\in\C, \|\nabla w\|_g\leq 1}\,\,\ell(\nabla w) <\infty\right \}.$$
 %

 \begin{lemma}\label{lem representation}
 $(L_\nabla^g)^*$ can be identified with (i.e.\ is isometrically isomorphic to) the quotient of $L^{g^*}$ by the relation
 \begin{equation}\label{identifica}
 \beta \,\R\, \bar{\beta}\,\, \iff \,\, \forall w\in\C:\, \iint [\beta(t,x)-\bar{\beta}(t,x)]'a(t,x)\nabla w(t,x)\mu_t(\dd x) \dd t =0,
\end{equation}  
when $L^{g^*}$ is given the ``operator norm'' as the dual of $L^g$, and the quotient space the derived norm topology.

In particular, if value$\eqref{eq dual original}<\infty$, then there is a unique equivalence class $[\beta]_\R$ such that
\begin{equation}
\label{rep of L}
L\psi\,\,=\,\,\iint \beta(t,x)'a(t,x)\psi(t,x)\mu_t(\dd x)\dd t,\,\,\,\, \mbox{ for all }\psi \in L_\nabla^{g}.
\end{equation}

 \end{lemma}
 
 \begin{proof}
 The subspace $$\textstyle M:=\{\beta \in L^{g^*}\mbox{ s.t. } \forall w\in\C:\, \iint \beta(t,x)'a(t,x)\nabla w(t,x)\mu_t(\dd x) \dd t =0\},$$
 is clearly closed. Notice that $\beta \,\R\, \bar{\beta} \iff \beta-\bar{\beta}\in M$. By classical results, the quotient space $L^{g^*}/\R$ is Banach with the norm $\|[\beta]_\R\|=\inf \{\|\beta + m\|_{g^*}:m\in M\}$. On the one hand, each equivalence class $[\beta]_\R$ defines an element of $(L_\nabla^g)^*$. On the other hand, if $\ell \in (L_\nabla^g)^*$, by Hahn-Banach theorem, $\ell$ can be extended by an $\tilde{\ell}\in (L^g)^*=L^{g^*}$ with $ \|\ell\|_{(L_\nabla^g)^*}=\|\tilde{\ell}\|_{g^*}$. By definition the function $\ell \mapsto [\tilde{\ell}]_\R$ is well-defined, surjective and linear. This function is also an isometry. Indeed, we have already obtained $\|[\tilde{\ell}]_\R\|\leq \|\ell\|_{(L_\nabla^g)^*}$ by the Hahn-Banach argument, whereas the converse inequality is trivial for the operator norm. The last statement is a consequence of the identification of $(L_\nabla^g)^*$ and Lemma \ref{lem extension functional}.
 \end{proof}

\medskip


Owing to the previous lemmas, we can finally say that the expression of the \emph{extended dual problem} \eqref{eq dual extendido}, given in Section \ref{main results}, is now rigorously defined. We have 

\begin{lemma}\label{lem dual ext}
The functional 
$$ \psi\in  L^g_\nabla \mapsto G(\psi):=\iint g(t,x,\sigma'(t,x)\psi(t,x))\mu_t(\dd x) \dd t $$
is convex and norm-continuous. As a consequence, the values of \eqref{eq dual original} and \eqref{eq dual extendido} coincide.
\end{lemma}

\begin{proof}
 Clearly $G$ is convex and finite, so we only need to show its local boundedness. Let $\psi$ given and take $\phi$ s.t.\ $\| \psi-\phi \|_g\leq 1/2$. By convexity, we find
$$G(\phi)\leq \frac{1}{2}G(2\psi)+\frac{1}{2}G(2[\phi-\psi])\leq \frac{1}{2}G(2\psi) +1/2,$$
since by assumption $G(0)=0$ so by convexity again $$G([\phi-\psi]/(1/2))\leq 1\,\,\mbox{ if }\,\,\|\psi-\phi \|_g\leq 1/2.$$
The second statement follows from the first one and the continuity in Lemma \ref{lem extension functional}.
\end{proof}

\begin{lemma}\label{lem integrabilidad}
For any $\psi\in L^g_\nabla$ we have that $(t,x)\mapsto \nabla g(t,x,\sigma'(t,x)\psi(t,x))\,\,\in L^{g^*}$.
\end{lemma}

\begin{proof}
Denote $f(t,x):= \nabla g(t,x,\sigma(t,x)'\psi(t,x))$. By definition of convex conjugates, we have
\begin{align*} & \iint g^*(t,x,f(t,x))\mu_t(\dd x) \dd t \\ = & -\iint g(t,x,\sigma(t,x)'\psi(t,x)) \mu_t(\dd x) \dd t + \iint \psi(t,x)'\sigma(t,x) f(t,x)\mu_t(\dd x) \dd t  ,\end{align*}
so finiteness of the l.h.s.\ is equivalent to that of the second term in the r.h.s., since $\sigma'$ is bounded. By convexity, $g(t,x,2y)\geq g(t,x,y)+ y'\nabla g(t,x,y)$. From Assumption ${\bf (C).5}$ on $g^*$ we can conclude that ${\bf (C).4}$ holds for $g$ instead. Thus, there is $c>0$ and $\alpha(\cdot,\cdot)$ non-negative and $\mu_t(\dd x)\dd t$-integrable such that 
\begin{equation}\label{eq estim grad}
y'\nabla g(t,x,y) \leq c g(t,x,y)+\alpha(t,x).
\end{equation}
 In particular, $\psi(t,x)'\sigma(t,x) f(t,x)\leq c g(t,x,\sigma(t,x)'\psi(t,x))+\alpha(t,x) $, so we conclude that the expressions above are finite as desired.  
\end{proof}

\begin{lemma}
We have 
\begin{equation}
\label{eq coercivity}
\lim_{\|\psi\|_g\to\infty}\frac{G(\psi)}{\|\psi\|_g}\,\, =\,\, +\infty,
\end{equation}
with $G$ as in Lemma \ref{lem dual ext}. Further, $G$ is directionally G\^ateaux differentiable, and for all $\psi \in L_\nabla^g,\,w\in\C$ we have:
\begin{equation}\label{eq gateaux}\textstyle
DG(\psi)(\nabla w)\,\,=\,\, \iint \nabla g(t,x,\sigma'(t,x)\psi(t,x))'\sigma'(t,x)\nabla w(t,x)\mu_t(\dd x)\dd t\,.
\end{equation}
\end{lemma}

\begin{proof}
Assumption ${\bf (C)}.4$ on $g^*$ implies Assumption ${\bf (C)}.5$ written on $g$ instead (for some $\ell >1$ and some integrable $H$). Applying this inequality repeatedly, one finds for each $p\in \mathbb{N}$ that $g(t,x,y\ell^p)\geq 2^p\ell^p g(t,x,y) - k\ell^p H(t,x) $. Let $\psi$ be s.t.\ $\|\psi\|_g\geq \ell^p$, then
\begin{multline*}
\frac{g(t,x,\sigma'\psi(t,x))}{\|\psi\|_g}= \frac{g(t,x,\sigma'\psi(t,x)\ell^p/\ell^p)}{\|\psi\|_g}\geq \frac{2^p\ell^p g(t,x,\sigma'\psi(t,x)/\ell^p)}{\|\psi\|_g} - \frac{k\ell^p H(t,x)}{\|\psi\|_g} \\
\geq 2^p g\left(t,x,\frac{\ell^p\sigma'\psi(t,x)}{\ell^p \|\psi\|_g}\right) - \frac{k\ell^p H(t,x)}{\|\psi\|_g} ,
\end{multline*}
since $g(t,x,\cdot)$ is convex and clearly null at $0$. Since $g$ is finite (by the superlinear growth of $g^*$) and convex in its last argument, it is a continuous function of  it. By monotone convergence, this proves $$\textstyle \iint g(t,x,\sigma'\psi(t,x)/\|\psi\|_g)\mu_t(\dd x)\dd t = 1,$$
so by the previous inequalities we find $$\textstyle G(\psi)/ \|\psi\|_g  = \iint\frac{g(t,x,\sigma'\psi(t,x))}{\|\psi\|_g}\mu_t(\dd x)\dd t \geq 2^p- k\ell^p \frac{ \iint H(t,x)\mu_t(\dd x)\dd t }{\|\psi\|_g  }.$$
Taking $\|\psi\|_g \to \infty$ and then $p\to\infty$ implies \eqref{eq coercivity}. As for the G\^ateaux differentiability, we must compute
$\textstyle \lim_{\epsilon\to 0}\frac{G(\psi+\epsilon \nabla w)-G(\psi)}{\epsilon}\, ,$
which is equal to
$$\textstyle\lim_{\epsilon\to 0}\,\iint \int_0^1\nabla g\bigl(t,x,\sigma'(t,x)[\psi(t,x)+\epsilon\theta\nabla w(t,x)]\bigr) '\sigma'(t,x)\nabla w(t,x)\dd\theta\mu_t(\dd x)\dd t.$$
But the innermost integral, as a function of $(t,x)$ converges a.s.\ when $\epsilon\to 0$ to $ \nabla g(t,x,\sigma'(t,x)\psi(t,x))'\sigma'(t,x)\nabla w(t,x)$. Applying the bound \eqref{eq estim grad} and the integrability result in Lemma \ref{lem integrabilidad} we may use dominated convergence to conclude.
\end{proof}


\section{No duality gap}\label{No gap}

For our main results in Section \ref{main results} it will be crucial to establish the equality between the Primal \eqref{eqprimal extendido} and the Dual \eqref{eq dual original} problems. We obtain this in the present section. So far we have kept the flow of marginals $\mu:=\{\mu_t\}_t$ fixed (see Assumption ${\bf (B)}$). For this part of the article we shall \emph{vary} this flow of marginals. Thus, we let $$\nu:=\{\nu_t\}_t \in C([0,T];\P(\RR^q)),$$ stand for a generic weakly continuous flow of measures with $\nu_0= m_0$, and use the notation $(P[\nu])$ and $(D[\nu])$ respectively for the Primal and Dual problem under such flow, in accordance to the notations used so far. For convenience, we write $P[\nu]$ and $D[\nu]$ for the value of these problems.

Let us define
\begin{align}
P^*[f] &:= \sup_{\nu\in C([0,T];\P(\RR^q)),\,\nu_0=m_0}\left \{ \iint f(t,x) \nu_t (\dd x) \dd t  -  P[\nu] \right\},\\
P^{**}[\nu] &:= \sup_{f \in \C}\left \{ \iint f(t,x)\nu_t(\dd x) \dd t  -  P^*[f] \right\}.
\end{align}

\begin{lemma}\label{lem multiple}
We have:
\begin{enumerate}
\item  $\{\nabla f:\, f\in C_b^{1,2}([0,T]\times\RR^q)\}\subset L^g_\nabla$.
\item Problem $(D[\nu])$ is equal to
\begin{align}\textstyle
\tag{D3$[\nu]$}
\label{eq dual 3}
\sup\limits_{h\in C_b^{1,2}([0,T]\times\RR^q)}&\Bigl \{ -\int h(T,x) \nu_T(\dd x) + \int h(0,x) m_0(\dd x) + \Bigr . \\& \left . \iint [  \L_t h(t,x)  - g(t,x,\sigma'(t,x)\nabla h(t,x))]\nu_t(\dd x)\dd t\right \},\notag
\end{align}
\item Problem $(D[\nu])$ is equal to 
\begin{align}\label{eq extended crucial}\tag{D4$[\nu]$} 
\sup_{f \in\C}\left \{ \iint f(t,x)\nu_t (\dd x) \dd t  -\int \, \left [E^g_0\left (  \int_0^T f(t,X_t)\dd t   \right )\right ](x_0)\, m_0(\dd x_0)   \right\},
\end{align}
where $E^g_0$ denotes the minimal supersolution operator.
\end{enumerate}
\end{lemma}

\begin{proof}

For Point 1 we follow the final part of the proof of \cite[Proposition 3.2]{CattLeo}. One first observes that $\{\nabla w:w\in\C\}$ is dense in $\{\nabla f:f\in C_b^{1,2}\}$ with respect to the weak topology $(\nabla f,\nabla w)=\iint \nabla f a \nabla w \dd\nu_t\dd t$. By Ascoli Theorem, this shows that $\{\nabla f:f\in C_b^{1,2}\}$ is in the closure of $\{\nabla w:w\in\C\}$ w.r.t.\ the weak topology $\sigma(L^g, L^{g^*})$. But by Mazur's Lemma this closure coincides with $L^g_\nabla$ and we conclude. 

We prove Point 2. Clearly $D3[\nu]\geq D[\nu]$. For the converse, we may assume $D[\nu]<\infty$. One verifies, for all $h\in C_b^{1,2}$, that 
$$ L(\nabla h) = -\int h(T,x) \nu_T(\dd x) + \int h(0,x) \nu_0(\dd x) +\iint  \L_t h(t,x)\nu_t(\dd x)\dd t, $$
by Point 1 and standard approximation arguments. We conclude by Lemma \ref{lem dual ext}.

Finally we prove Point 3. Let $f\in\C$ and observe that $$E^g_0\left(\int_0^Tf(t,X_t)\dd t\right)=E^{f+g}_0\left  ( 0 \right ),$$ 
both being functions of $X_0$. By \cite[Theorem 5.2]{markovminsupersol}, which is applicable thanks to \cite[Proposition 3.5.(iv)]{markovminsupersol} and our Assumotion ${\bf (C)}.3$, the above values equal $u(0,X_0)$, where $u$ is the minimal viscosity supersolution of
$$\L_t u +g(t,x,\sigma'(t,x)\nabla u(t,x))+f(t,x)\leq 0,\,\,\, u(T,\cdot)\geq 0. $$
Let $\Phi$ be a sufficiently smooth function\footnote{For instance $\Phi(t,x):=C-t\sup_{s,y}|f(s,y)|$ with $C\geq T\sup_{s,y}|f(s,y)|$ fulfils this.} such that $\L_t\Phi(t,x) + g(t,x,\sigma'\nabla \Phi(t,x))+ f(t,x)\leq 0$ and $\Phi(T,\cdot)\geq 0$. Then
$$\iint f(t,x)\nu_t(\dd x) \dd t \leq  \int\Phi(T,x)\nu_T(\dd x) - \iint[ \L_t\Phi(t,x) + g(t,x,\sigma'\nabla \Phi(t,x))]\nu_t (\dd x) \dd t. $$
By \cite[Theorem 5]{FlemingVermes}, we actually have $u(0,X_0)=\inf \Phi(0,X_0)$ $m_0$-a.s., namely that the minimal viscosity supersolution is the infimum over classical supersolutions. From this and the previous considerations, we obtain for each $\epsilon >0$ the existence of $\Phi=\Phi^\epsilon$ such that
\begin{align*}
 & \iint f(t,x)\nu_t (\dd x) \dd t  -\int\,E_0^g\left (  \int_0^T f(t,X_t)\dd t   \right )\,\dd m_0(X_0)  \\ =&  -\int u(0,x)m_0(\dd x) +  \iint f(t,x)\nu_t(\dd x) \dd t\\\leq & \epsilon - \int \Phi(0,x)m_0(\dd x)  + \int\Phi(T,x)\nu_T(\dd x) -  \iint[ \L_t \Phi(t,x)+g(t,x,\sigma'\nabla \Phi(t,x)) ]\nu_t (\dd x) \dd t\\
 \leq &\epsilon+  D3[\nu]\\ = &\epsilon+ D[\nu].
\end{align*}
The last inequality comes from taking $-h$ instead of $h$ in \eqref{eq dual 3}.
So $D4[\nu]\leq D[\nu]$. The converse inequality follows by taking, for each $w\in\C$, $f^w:=\L_t w - g(\sigma'\nabla w)$, and elementary approximation arguments. 
\end{proof}

\begin{lemma}\label{lem P*,D,P**}
We have
\begin{align}
P^*[f]&= \int E^g_0\left( \int_0^T f(t,X_t)\dd t\right)(x_0)\dd m_0(x_0),\mbox{  for }f\in C_b([0,T]\times \RR^q), \label{P*=-j}
\end{align}
and
\begin{align}
\label{eq D=P**}
D[\nu] & = \sup_{f\in \C} \left\{ \iint f(t,x)\nu_t(\dd x)\dd t   -\int E^g_0\left( \int_0^T f(t,X_t)\dd t\right)(x_0)\dd m_0(x_0) \right\},\\
&= P^{**}[\nu].\label{eq P** aux}
\end{align}
\end{lemma}

\begin{proof}
We start with \eqref{P*=-j}. To wit
\begin{align*}
P^*[f]&= \sup_{\nu\in C([0,T];\P(\RR^q)),\,\nu_0= m_0}\left \{ \iint f(t,x)\nu_t(\dd x)\dd t   -  P[\nu] \right\}\\
\\
&= \sup_{\substack{\nu\in C([0,T];\P(\RR^q)),\,\nu_0= m_0,\\ \QQ \in \Q(\nu)}} \EE^\QQ \left [ \int_0^T  f(t,X_t)  \dd t  - \int_0^T g^*(t,X_t,\sigma'(t,X_t)\beta^\QQ_t)\dd t  \right ]\\
&= \sup_{\substack{ \QQ}}\,\, \left\{\EE^\QQ \left [ \int_0^T  f(t,X_t)  \dd t  \right ]- \tilde{I}(\QQ) \right \}\\
&=\int E^g_0\left( \int_0^T f(t,X_t)\dd t\right)(x_0)\dd m_0(x_0),
\end{align*}
by Lemma \ref{identificacionconjugadas}. This shows that 
\begin{align*} P^{**}[\nu]& = \sup_{f \in\C}\left \{ \iint f(t,x)\nu_t(\dd x)\dd t  -\int E^g_0\left( \int_0^T f(t,X_t)\dd t\right)(x_0)\dd m_0(x_0) \right\},
\end{align*}
which yields \eqref{eq P** aux}, whereas \eqref{eq D=P**} follows by Lemma \ref{lem multiple}.
\end{proof}

\begin{proposition}\label{prop no d gap}
We have $P[\nu]= P^{**}[\nu]=D[\nu]$, i.e.\ there is no duality gap. 
\end{proposition}

\begin{proof}
It is easy to see that $P[\nu]\geq P^{**}[\nu]$ and that $P[\cdot]$ is convex. In light of Lemma \ref{lem P*,D,P**}, to obtain no duality gap it suffices to prove $P[\nu]=P^{**}[\nu]$.  We now establish that $P[\cdot]$ is lower-semicontinuous in an appropriate sense. Let $\left\{\nu^\alpha\right\}$ be a net in $C([0,T];\P(\RR^q))$ for which $P[\nu^\alpha]\leq k$ and
\begin{equation}\label{eq point}\forall f\in\C:\,\, \iint f(t,x)\nu^\alpha_t(\dd x)\dd t \to \iint f(t,x)\nu_t(\dd x)\dd t ,\end{equation}
for some $\nu\in C([0,T];\P(\RR^q))$; we may assume all this functions start at $m_0$ at time zero. By Lemma \ref{lem tightness -} we have that $P[\nu^\alpha]=\tilde{I}(\QQ^\alpha)\leq k$ for unique probability measures $\QQ^\alpha\in\Q(\nu^\alpha)$. By Lemma \ref{lem folk} the family $\{\QQ^\alpha\}$ is tight. Let $\QQ$ be any accumulation point. For ease of notation we still index the subnet accumulating into $\QQ$ by the same indices. By the lower semicontinuity of $\tilde{I}$ given in Lemma \ref{lem folk}, we obtain $\tilde{I}(\QQ)\leq k$. On the other hand, for each $f\in\C$ we have 
\begin{align*}\EE^\QQ\left[ \int_0^Tf(t,X_t)\dd t\right]=&\lim \EE^{\QQ^\alpha}\left[\int_0^Tf(t,X_t)\dd t\right] \\ = & \lim  \iint f(t,x)\nu^\alpha_t(\dd x)\dd t \\ = & \iint f(t,x)\nu_t(\dd x)\dd t .\end{align*} 
Now take $F$ a smooth function on $\RR^q$ with bounded support, $\bar{t}\in (0,T)$ and $m^n$ a sequence of smooth functions of time converging monotonically (hence uniformly) to ${\bf 1}_{(\bar{t},\bar{t}+\epsilon)}$. Take $f(t,x)=m^n(t)F(x)$. By monotone convergence and the above equality, we get
 $$\EE^{\QQ}\left[\frac{1}{\epsilon}\int_{\bar{t}}^{\bar{t}+\epsilon}F(X_t)\dd t\right]=\frac{1}{\epsilon}\int_{\bar{t}}^{\bar{t}+\epsilon}\int_{\RR^q} F(x)\nu_t(\dd x) \dd t.$$
 By dominated convergence we get as $\epsilon\to 0$ that
  $$\EE^\QQ[F(X_{\bar{t}})]=\int F(x)\nu_{\bar{t}}(\dd x),$$
  since $t\mapsto \nu_t$ is weakly continuous. This identity must also hold for $F$ continuous bounded by further approximation arguments. The limiting cases of $\bar{t}\in\{0,T\}$ follow taking limits, as $t\mapsto \nu_t$ is weakly continuous. Therefore $\QQ$ has $\nu$ as its marginal flow. Since $\tilde{I}(\QQ)<\infty$ we conclude that $\QQ\in\Q(\nu)$, therefore is feasible for $(P[\nu])$, and we deduce $P[\nu]\leq k$ as desired.

  Wrapping up, we obtained that $P[\cdot]$ is convex and lower semicontinuous w.r.t.\ pointwise convergence on $\C$ (i.e.\ in the sense of \eqref{eq point}). By construction $P^{**}[\cdot]$ is the greatest minorant of $P[\cdot]$ having these properties, so we conclude $P^{**}[\cdot]=P[\cdot]$. 
\end{proof}

\section{Proof of the main result}\label{proof main result}

The following is a crucial result for this part:
 
 \begin{proposition}\label{prop equiv}
 value$\eqref{eq dual original}<\infty$ is equivalent to the existence of some $\Psi\in \L^g_\nabla$ such that for all $w\in \C$: 
 \begin{equation} \iint [ \L_t w(t,x)- \nabla g(t,x,\sigma'(t,x){\Psi}(t,x))'\sigma'(t,x)\nabla w(t,x)]\mu_t(\dd x) \dd t\,\, =\,\,0
 \label{eq FP}
 \end{equation}
 When this holds, then $\Psi$ is an optimizer for \eqref{eq dual extendido} and
\begin{equation}
{value}\eqref{eq dual original} = \mbox{value}\eqref{eq dual extendido} = \iint g^*\big (t, x, \nabla g(t ,x,\sigma'(t,x)\Psi(t,x) )\big)    \mu_t(\dd x) \dd t. \label{eq energy}
\end{equation}
This $\Psi$ is $\mu_t(\dd x)\dd t$-a.s.\ unique, 
%
%
and we further have that $$(t,x)\mapsto(\sigma')^{-1}(t,x)\nabla g(t,x,\sigma'(t,x)\Psi(t,x))$$ is the unique (up to equivalence class) representative of $L$ (cf.\ Lemma \ref{lem representation}).
 \end{proposition}
 
 \begin{proof} First we assume value$\eqref{eq dual original}<\infty$.  Clearly \eqref{eq mod continuity} implies that $L(\psi)\leq \{1+\mbox{value}\eqref{eq dual original}\}\,\|\psi\|_g$. We thus find
\begin{align*} L(\phi) -\iint g(t,x,\sigma'(t,x)\psi(t,x)),\mu_t(\dd x) \dd t  \leq  \|\psi\|_g\left [   1+\mbox{value}\eqref{eq dual original} - \frac{G(\psi)}{\|\psi\|_g}  \right ],
\end{align*}
in the notation of Lemma \ref{lem dual ext}. Using \eqref{eq coercivity} we find that the l.h.s.\ goes to $-\infty$ if we let $\|\psi\|_g\to \infty$. We deduce that computing \eqref{eq dual extendido} can be done over a fixed ball in $L_\nabla^g$. But $L^g$ is reflexive by Lemma \ref{lem Kozek}, so balls in $L_\nabla^g$ are weakly compact. The objective function of the extended dual problem being concave continuous (see Lemmata \ref{lem dual ext},\ref{lem extension functional}), it is also weakly upper semi-continuous. We conclude the existence of an optimizer for the extended dual problem. Let $\Psi$ denote any optimizer and $\nabla w$ any ``direction''. The optimality of $\Psi$ easily yields
$$\iint \L_t w(t,x)\mu_t(\dd x) \dd t - DG(\Psi)(\nabla w)=0.$$ 
Thanks to \eqref{eq gateaux} this proves \eqref{eq FP}, which further implies for all $\psi\in L^g_\nabla$: 
\begin{equation}
L(\psi)=\iint \nabla g(t,x,\sigma'(t,x){\Psi}(t,x))'\sigma'(t,x)\psi(t,x)\mu_t(\dd x) \dd t. \label{eq FOC general}
\end{equation}

For the converse direction, we observe that \eqref{eq FP} combined with Lemma \ref{lem integrabilidad}, allows to perform the continuous extension $L$ of Lemma \ref{lem extension functional}. Thus one can define the extended dual problem anew. By \eqref{eq FP} and continuity, the extended dual becomes
$$\sup_{\psi\in L_\nabla^g} \iint [\nabla g(t,x,\sigma'(t,x){\Psi}(t,x))'\sigma'(t,x)\psi(t,x) - g(t,x,\sigma'(t,x)\psi(t,x))]\mu_t(\dd x) \dd t ,$$
which is bounded above by the r.h.s.\ of \eqref{eq energy} by convex conjugacy. This bound is finite by Lemma \ref{lem integrabilidad}, so a fortriori the non-extened dual problem is finite as desired. 

For \eqref{eq energy}, substitute \eqref{eq FOC general} into the extended dual (evaluated at $\Psi$), obtaining
$$\mbox{value}\eqref{eq dual extendido} = \iint [\nabla g(t,x,\sigma'(t,x){\Psi}(t,x))'\sigma'(t,x)\Psi(t,x) - g(t,x,\sigma'(t,x)\Psi(t,x))]\mu_t(\dd x) \dd t,$$ 
which in effect yields \eqref{eq energy} due to the conjugacy relationship. The remark on uniqueness of $\Psi$ follows from the differentiability of $g^*$, which implies the strict convexity of $g$. The last statement follows by \eqref{eq FP} (equiv.\ \eqref{eq FOC general}), which implies that the given element does represent $L$ acting on $L_\nabla^g$, and Lemma \ref{lem representation}, implying uniqueness of such representative up to equivalence class.
%
%
 \end{proof}

%
%

\quad

We can now prove the main structural result of the article.
\medskip

\begin{proof}[Proof of Theorem \ref{Thm Main}]
Absence of duality gap was obtained in Proposition \ref{prop no d gap}. From now on we assume $\mbox{value}\eqref{eqprimal extendido}<\infty$. The existence of a unique optimal $\Psi$ is given by Proposition \ref{prop equiv}.
The existence of a (unique) primal optimizer $\QQ$ was established in Lemma \ref{lem tightness -}. We proceed to show that this $\QQ$ must have the desired property.

Since $\QQ\in\Q(\mu)$, we have for some drift $\beta$:
$$\EE^\QQ\left [ \int_0^T (\L_t+\beta_t 'a\nabla)w(t,X_t)\dd t \right ]=0,\,\,\forall w\in\C.$$
%

Let $\bar{\beta}(t,x)=\EE^\QQ[\beta_t|X_t=x]$,
%
%
so that obviously
\begin{equation}\label{eq unicidad}
-\iint \L_t w(t,x)\mu_t(\dd x)\dd t=\iint  \bar{\beta}(t,x)' a(t,x)\nabla w(t,x)\mu_t(\dd x)\dd t,\,\,\forall w\in\C.
\end{equation}
Plugging in this representation of the l.h.s.\ into the dual problem, and using the Young-Fenchel inequality we obtain
\begin{equation}\label{eq bound dual}
\mbox{value}\eqref{eq dual original} \leq \iint g^*(t,x,\sigma'\bar{\beta}(t,x))\mu_t(\dd x) \dd t .
\end{equation}
By Jensen's inequality, the fact that $\QQ$ has marginals $\{\mu_t\}_t$, the above equation and \eqref{eq energy}, we deduce
\begin{align}
\mbox{value}\eqref{eqprimal extendido} = & \EE^{\QQ}\left[\int_0^T g^*(t,X_t,\sigma'(t,X_t)\beta_t)\dd t\right] \notag \\ \geq & \EE^{\QQ}\left[\int_0^T g^*(t,X_t,\sigma'(t,X_t)\bar{\beta}(t,X_t))\dd t\right] \notag \\ = &  \iint g^*(t,x,\sigma'\bar{\beta}(t,x))\mu_t(\dd x) \dd t \notag  \\ \geq & \mbox{value}\eqref{eq dual original} \notag  \\ =&
\iint g^*(t, x,\, \nabla g(t ,x,\sigma'(t,x)\Psi(t,x) )\,)   \mu_t(\dd x) \dd t. \label{eq cadena crucial}
\end{align}
By no duality gap, the above inequalities are actual equalities. Since $g^*$ is stricly convex, this shows that 
\begin{equation}\label{eq final 1}
\QQ\times \dd t-a.s.\,\,\, \beta_t(X) = \bar{\beta}(t,X_t). 
\end{equation}
On the other hand \eqref{eq FP} with \eqref{eq cadena crucial} show that the problem
$$\inf\limits_{\substack{k(\cdot,\cdot) \text{ s.t. }\forall w\in \C :\\  \iint  \L_t w(t,x)\mu_t(\dd x)\dd t=\iint  k(t,x)' a(t,x)\nabla w(t,x)\mu_t(\dd x)\dd t}} \iint g^*(t,x,\sigma'k(t,x))\mu_t(\dd x) \dd t,$$
has $-\bar{\beta}(\cdot,\cdot)$ and $ (\sigma')^{-1}(\cdot,\cdot)\nabla g(\cdot ,\cdot,\sigma'(\cdot,\cdot)\Psi(\cdot,\cdot) )$ as feasible elements, where the latter is optimal. Indeed, \eqref{eq bound dual} holds also for any $k(\cdot,\cdot)$ participating in the infimum above. Again by strict convexity of $g^*$ and the equality in \eqref{eq cadena crucial} we find that
\begin{equation}\label{eq final 2}
\mu_t(\dd x)\times\dd t-a.s.\,\,\, \bar{\beta}(t,x)=- (\sigma')^{-1}(t,x)\nabla g(t ,x,\sigma'(t,x)\Psi(t,x) ).
\end{equation}
Calling $\Lambda\subset \RR^q\times [0,T]$ the set on which \eqref{eq final 2} fails, we have 
$$\textstyle 0=\int\int {\bf 1}_\Lambda(t,x)\mu_t(\dd x)\dd t = \int_0^T \EE^\QQ[  {\bf 1}_\Lambda(t,X_t) ]\dd t,$$
showing that
\begin{equation}\label{eq final 3}
\textstyle
\QQ\times \dd t-a.s.\,\,\, \bar{\beta}(t,X_t)=- (\sigma')^{-1}(t,X_t) \nabla g(t ,X_t,\sigma'(t,X_t)\Psi(t,X_t) ).
\end{equation}
Putting \eqref{eq final 1} and \eqref{eq final 3} together, we find \eqref{formula crucial}. From this \eqref{formula crucial abs} is also clear.
\end{proof}

\section{Proofs of the main corollaries}
\label{sec: ex}

We prove here Corollaries \ref{coro mark} and \ref{coro univ}. Most of the effort is devoted to the construction of counterexamples. \\

\begin{proof}[Proof of Corollary \ref{coro mark}]
We show in Section \ref{sec ex mark} below an example of an optimizer without the Markov property. Let us now assume the sufficient condition in the statement, so we have $$\frac{\dd \QQ}{\dd \PP}:= \,\E\left ( -\int \nabla g(t,X_t,a(t,X_t){\Psi}(t,X_t))'\dd M_t   \right )_T.$$
The argument is now as in  \cite[Theorem 12]{Zheng_tightness}. Let us call $Z_t $ the associated density process, which is a true $\PP$-martingale. Let $s\leq t$, $F$ be an $\F_s$-measurable bounded function and $f:\RR^q\to\RR$ Borel bounded. Then
\begin{align*}
\EE^\QQ[F\, f(X_t)] &= \EE^{\PP}\left[F Z_s \, \E\left( -\int_s^t \nabla g(t,X_t,a(t,X_t){\Psi}(t,X_t))'\dd M_t \right) f(X_t)\right] \\ &= \EE^{\PP}\left[F Z_s \, \EE^{\PP}\left[\, \E\left( -\int_s^t \nabla g(t,X_t,a(t,X_t){\Psi}(t,X_t))'\dd M_t \right) f(X_t)\,  \Bigl | \Bigr . \F_s\,\right ]\,\right]
\\ &= \EE^{\QQ}\left[F \, \EE^{\PP}\left[\, \E\left( -\int_s^t \nabla g(t,X_t,a(t,X_t){\Psi}(t,X_t))'\dd M_t \right) f(X_t)\,  \Bigl | \Bigr . X_s\,\right ]\,\right].
\end{align*}
The last equality by the Markov property under $\PP$ and the fact that nothing in the stochastic exponential there depends on $\{X_r:r\leq s\}$. This finishes the proof.
\end{proof}\\

\begin{proof}[Proof of Corollary \ref{coro univ}]
The assertion in one-dimension is fully analogous to \cite[Proposition 5.2 and Remark 5.5]{CattLeo_AIHP}. Indeed, there is actually at most one Markovian measure with the given marginals and with an integrable drift. For higher dimensions, see the example in Section \ref{sec ex univ} below.
\end{proof}

\subsection{Non-Markovian optimal solution}\label{sec ex mark}

The (counter)example is based on the $\text{Bessel}(\delta)$ process, with dimension parameter $1<\delta<2$, equiv.\ index $\nu=\delta/2-1\in(-1/2,0)$; see \cite[Appendix I.21]{BoSa02}. From the expression of the probability density function $p^\nu$ of this process, and the asymptotics of Bessel functions, we have that
$$p^\nu_t(x,y)y^{-2\nu-1}$$
is bounded away from zero and infinity, for each $t>0$
 fixed and $y$ in a neighbourhood of the origin. Therefore 
$$\int y^{-p}p_t^\nu(x,y)\dd y<\infty,$$
 as soon as $1<p<{2\nu+2}$. Denoting by $X$ the Bessel process described, it is an easy consequence of scaling and the finite integral above, that
 $$\EE^{x_0}\left[ \int_0^1 \frac{1}{|X_t|^p}\dd t \right]<\infty.$$
 We recall that $X$ started at $x_0>0$ satisfies $$dX_t = \frac{\delta-1}{2 X_t}\dd t+\dd W_t, \,\,X_0=x_0,$$
 and is in fact the unique positive solution of this SDE. Actually, the origin is instantaneously reflected by this process. Denote $\tau$ the first time that $X$ touches the origin. We now construct a second process, as in \cite[Example 3.10]{ChEn05}, by
 $$Y_t=X_t \,\,\text{if}\,\,t\leq \tau,$$
 and $Y_t= \text{sign}(X_{\tau/2}-1)X_t$, for $t>\tau$. One can see that $Y$ is a weak solution of the same SDE as $X$, and has the same finite moment
  $$\EE\left[ \int_0^1 \frac{1}{|Y_t|^p}\dd t \right]<\infty.$$
  On the other hand $Y$ is clearly non-Markovian. Denoting $\mu_t:=\text{Law}(Y_t)$, and taking $\PP$ the Wiener measure, we have
  
  \begin{lemma}
  $\text{Law}(Y)$ is the unique optimizer of our primal problem for the cost $g^*(t,x,b)=|b|^p/p$ and the marginals $\{\mu_t\}_t$, with finite optimal cost if $1<p<{2\nu+2}$. In particular, solutions to our primal problem can fail to have the Markov property even if the value of the problem is finite.
  \end{lemma}
  \begin{proof} We have $g(z)=|z|^q/q$ with $q$ the H\"older conjugate of $p$.
  By the first order conditions of the dual problem problem, and the fact that $\nabla g(z)=\text{sign}(z)|z|^{q-1}$, it is easy to guess that 
  $$\Psi(t,x)=-\left( \frac{\delta-1}{2}|x|^{-1}\right )^{\frac{1}{q-1}}\text{sign}(x),$$
  is the dual optimizer. Indeed, to see that $\Psi$ is an $L^q(\dd\mu_t\dd t)$-limit of gradients, we just consider $w_n(x)=-\frac{\delta-1}{2}\times\frac{q-1}{q-2}\left(\frac{2}{\delta-1}|x|+n^{-1}\right)^{\frac{q-2}{q-1}}$, take gradients, and use dominated convergence.
  \end{proof}

\subsection{Non-universality of optimal solution}\label{sec ex univ}

We shall see that the optimizer can depend on the cost criterion.
%
%
Let $q=2$. For simplicity we shall consider a ``stationary'' case. We do so only to spare the reader with the heavier computations needed for the ``non-stationary'' analogue argument. The cost to pay is that the marginal distributions ($\mu_t$) must be $\sigma-$finite measures.

Let $B:\RR^2\to \RR$ be twice differentiable with bounded support. We take $$dX_t = \nabla B(X_t)\dd t + \dd W_t,$$
with initial condition $X_0$ distributed like two-dimensional Lebesgue measure, that is $\mbox{Law}(X_0) = \lambda^2$. We denote by $\PP$ the law of the unique strong solution of this SDE. We denote by $\QQ^{ent}$ th law of stationary (i.e.\ reversible) Brownian motion, that is Brownian motion with initial (and stationary) distribution $ \lambda^2$. Let us take $\mu_t = \lambda^2$ for all $t$, so the $t$-marginals of $\QQ^{ent}$ are precisely $\mu_t$. It is easy to see that
$$d\QQ^{ent}/d\PP = \exp\left\{ -\int_0^T \nabla B(X_t)'\dd W_t -\frac{1}{2}\int_0^T|\nabla B(X_t)|^2 \dd t \right \},$$
and that $\QQ^{ent}$ is optimal for the entropy minimization (primal) problem
$$\inf \left\{ \EE^\QQ\left[ \int_0^T \|\beta^\QQ\|^2\dd t \right]:\, d\QQ/d\PP=\E\left(-\int \beta' \dd W\right)_T,\, \QQ\circ X_t^{-1}=\lambda^2\,\, \mbox{ for all }t  \right \}.$$
Indeed, taking $\beta^B(t,X):=\nabla B(X_t)$ ensures producing the correct marginals, provides finite entropy, and has to be an optimal choice being a gradient (for instance by first order conditions, or see previous sections with $g$ quadratic). 

We now claim that different cost criteria than the above quadratic one may yield different optimizers. Consider
$$\inf \left\{ \EE^\QQ\left[ \int_0^T \|\beta^\QQ\|^{3}\dd t \right]:\, \dd \QQ/\dd \PP=\E\left(-\int \beta '\dd W\right)_T,\, \QQ\circ X_t^{-1}=\lambda^2\,\, \mbox{ for all }t  \right \}.$$
Observe that the power cost $g^*(\cdot):= (\cdot)^{3}$ satisfies our assumptions and that $\QQ^{ent}$ is feasible and produces a finite value for this cost criterion. We also have $g(\cdot)= \frac{2}{3}(\cdot)^{3/2}$. The optimizer for this problem has the structure
$$\frac{\dd \bar{\QQ}}{\dd \PP}=\E\left ( -\int\nabla g(\Psi(t,X_t))'\dd W_t \right )_T =  \E\left ( -\int\frac{\Psi(t,X_t)'}{\sqrt{\|\Psi(t,X_t)\|}}\dd W_t \right )_T ,$$
for $\Psi$ a solution to the dual problem, and so a limit of gradients. We want to give conditions so that $\QQ^{ent}\neq \bar{\QQ}$. For the sake of the argument let us assume now that $\Psi =\nabla w$ for $w$ suitable smooth. So we want to ensure the \emph{impossibility} of $$\nabla B = \frac{\nabla w(t,X_t)}{\sqrt{|\nabla w(t,X_t)|}}.$$  Taking norms on both sides we get $\|\nabla B\|=\sqrt{\|\nabla w\|}$, so we explore instead 
\begin{align}
\|\nabla B\|\nabla B = \nabla w. \label{eq B w}
\end{align}
The argument is simple now. For the r.h.s.\ we know, no matter who $w$ may be, that $$\partial_y\partial _x w = \partial_y(\mbox{1st coordinate of the r.h.s.})= \partial_x\partial_y w = \partial_x(\mbox{2nd coordinate of the r.h.s.}).$$ But by \eqref{eq B w} one computes that this is possible only if $$\partial_xB[\partial_yB\partial^2_{yy} B + \partial_xB \partial^2_{xy}B ]= \partial_yB[\partial_xB\partial^2_{xx} B + \partial_yB \partial^2_{xy}B ] .$$ So choosing $B$ such that this does not occur (for instance take $B(x,y)=p(x)q(y)$ with $p,q$ non-trivial, smooth and with bounded support) we see that there is no smooth  $w$ for which \eqref{eq B w} may hold. The general case with $\Psi$ is similar, by integration by parts and from the fact that $\Psi$ is a limit of actual gradients. In such case, no matter who the dual optimizer is, the induced optimal measure will not have a stochastic logarithm equal to $\nabla B(t,X_t)$.

\bibliographystyle{plain}
\bibliography{bibliotesis}

\end{document}